\documentclass[reqno]{amsart}
\pdfoutput=1
\usepackage[alphabetic]{amsrefs}
\usepackage{amsmath}
\usepackage{amsfonts}
\usepackage{amssymb}
\usepackage{enumerate}
\usepackage{amstext}
\usepackage{amsbsy}
\usepackage{amsopn}
\usepackage{bbm,amsthm}
\usepackage{amscd}
\usepackage[pdftex]{color}
\usepackage{amsxtra}
\usepackage{upref}
\usepackage{epstopdf}
\usepackage{graphicx,color}

\DeclareFontFamily{OML}{rsfs}{\skewchar\font'177}
\DeclareFontShape{OML}{rsfs}{m}{n}{ <5> <6> rsfs5 <7> <8> <9> rsfs7
  <10> <10.95> <12> <14.4> <17.28> <20.74> <24.88> rsfs10 }{}
\DeclareMathAlphabet{\mathfs}{OML}{rsfs}{m}{n}

\newtheorem{theorem}{Theorem}
\newtheorem{lemma}[theorem]{Lemma}

\newtheorem{corollary}[theorem]{Corollary}

\theoremstyle{definition}

\theoremstyle{remark}
\newtheorem{remark}[theorem]{\bf Remark}

\numberwithin{equation}{section}
\numberwithin{theorem}{section}

\newcommand{\intav}[1]{\mathchoice {\mathop{\vrule width 6pt height 3 pt depth  -2.5pt
\kern -8pt \intop}\nolimits_{\kern -6pt#1}} {\mathop{\vrule width
5pt height 3  pt depth -2.6pt \kern -6pt \intop}\nolimits_{#1}}
{\mathop{\vrule width 5pt height 3 pt depth -2.6pt \kern -6pt
\intop}\nolimits_{#1}} {\mathop{\vrule width 5pt height 3 pt depth
-2.6pt \kern -6pt \intop}\nolimits_{#1}}}

\newcommand{\intavl}[1]{\mathchoice {\mathop{\vrule width 6pt height 3 pt depth  -2.5pt
\kern -8pt \intop}\limits_{\kern -6pt#1}} {\mathop{\vrule width 5pt
height 3  pt depth -2.6pt \kern -6pt \intop}\nolimits_{#1}}
{\mathop{\vrule width 5pt height 3 pt depth -2.6pt \kern -6pt
\intop}\nolimits_{#1}} {\mathop{\vrule width 5pt height 3 pt depth
-2.6pt \kern -6pt \intop}\nolimits_{#1}}}

\newcommand{\ve}{\varepsilon}

\renewcommand{\P}[1]{{\mathbb{P}}\left[{#1}\right]}

\newcommand{\EE}[2]{{\mathbb{E}}\left[{#1}|{#2}\right]}

\begin{document}

\title[Graph-based P\'{o}lya's urn: the linear case]{Graph-based P\'{o}lya's urn: completion of the linear case}

\author[Yuri Lima]{Yuri Lima}
\thanks{The author was supported by the Brin Fellowship.}
\address{Department of Mathematics, University of Maryland, College Park, MD 20742, USA.}
\email{yurilima@gmail.com}


\date{\today}


\begin{abstract}
Given a finite connected graph $G$, place a bin at each vertex. Two bins are called a pair if they
share an edge of $G$. At discrete times, a ball is added to each pair of bins.
In a pair of bins, one of the bins gets the ball with probability proportional to its current
number of balls. This model was introduced in \cite{benaim2013generalized}. When $G$ is not balanced bipartite, the
proportion of balls in the bins converges to a point $w(G)$ almost surely \cite{benaim2013generalized,chen2013generalized}.

We prove almost sure convergence for balanced bipartite graphs: the possible limit
is either a single point $w(G)$ or a closed interval $\mathfs J(G)$.
\end{abstract}

\maketitle

\section{Introduction}

Let $G=(V,E)$ be a finite connected graph with $V=[m]=\{1,\ldots,m\}$ and $|E|=N$, and place a bin at vertex $i$ with $B_i(0)\ge 1$ balls.
Consider a random process of adding $N$ balls to the bins at each step, according to the following
law: if the numbers of balls after step $n-1$ are $B_1(n-1),\ldots,B_m(n-1)$, step $n$ consists
of adding, to each edge $\{i,j\}\in E$, one ball either to $i$ or to $j$, and the probability that the ball
is added to $i$ is
\begin{align}\label{transition-probability}
\P{i\text{ is chosen among }\{i,j\}\text{ at step }n}=\dfrac{B_i(n-1)}{B_i(n-1)+B_j(n-1)}\,\cdot
\end{align}
We call this model a {\em graph-based P\'olya's urn}.

\medskip
Call $G$ {\em bipartite} if there is a partition $V=A\cup B$ such that
for every $\{i,j\}\in E$ either $i\in A,j\in B$ or $i\in B,j\in A$. If $\#A=\#B$ we call $G$ {\em balanced bipartite},
and if $\#A\neq\#B$ we call it {\em unbalanced bipartite}.   

Let $N_0=\sum_{i=1}^m B_i(0)$ be the initial number of balls, let
$x_i(n)=\frac{B_i(n)}{N_0+nN}$, $i\in[m]$, and let $x(n)=(x_1(n),\ldots,x_m(n))$, the proportion of balls in the bins
after step $n$.

\begin{theorem}\label{main-theorem}
If $G$ is a finite, connected, balanced bipartite graph, then there is a closed interval $\mathfs J=\mathfs J(G)$
such that $x(n)$ converges to a point of $\mathfs J$ almost surely.
\end{theorem}

In some cases $\mathfs J$ is a singleton. When it is not, the point to which $x(n)$ converges 
depends on the realization of the process $x(n)$. 

\medskip
Graph-based P\'olya's urns were introduced in \cite{benaim2013generalized}:
for a fixed $\alpha>0$, one of the bins gets the
ball with probability proportional to the $\alpha$ power of its current number of balls.
When $\alpha=1$ the model is (\ref{transition-probability}),
hence we call it the {\em linear case}.

\medskip
Graph-based P\'olya's urns extend the classical P\'olya's urn and many of its variants, see \cite{pemantle2007survey}.
E.g. if $G$ is the complete graph with $m$ vertices then the model is a P\'olya's urn with $m$ colors.

\medskip
$\mathfs J$ depends on the structure of the graph. If $G$ is not bipartite then $\mathfs J$ is 
a singleton \cite{benaim2013generalized}. This was extended to unbalanced bipartite graphs \cite{chen2013generalized}.
The remaining case, when $G$ is balanced bipartite, was conjectured in \cite[Conjecture 5.4]{chen2013generalized}.
Theorem \ref{main-theorem} confirms it, and completes the description of possible limits.

\begin{corollary}\label{main-corollary}
If $G$ is a finite, connected graph, then there is a closed interval $\mathfs J=\mathfs J(G)$
such that $x(n)$ converges to a point of $\mathfs J$ almost surely. If $G$ is not balanced bipartite, 
then $\mathfs J$ is a singleton.
\end{corollary}

Additionally to being natural generalizations of P\'olya's urns, graph-based P\'olya's urns model some competing networks
\cite{benaim2013generalized}: Imagine there are 3 companies,
denoted by M, A, G. Each company sells two products.
M sells {\sf OS} and {\sf SE}, A sells {\sf OS} and {\sf SP}, G sells {\sf SE} and {\sf SP}.
Each pair of companies compete
on one product. The companies try to use their global size and reputation to boost sales.
Which company will sell more products in the long term? The interaction between the companies form a
triangular network: a vertex represents a company and an edge represents a product. Under further simplifications, graph-based
P\'olya's urns describe in broad strokes the long-term evolution of such competition.

\medskip
Another example comes from a repeated game in which agents improve their
skill by gaining experience \cite{skyrms2000dynamic}. The interaction network between agents is modeled by a graph.
At each round a pair is competing for a ball. A competitor improves his skill with time,
and the number of balls in his bin represents his skill level.
See \cite[\S1]{benaim2013generalized} and references therein for more applications.

\medskip
The sequence $x(n)$ is a {\em stochastic approximation algorithm}. These are small perturbations of a vector field.
In many cases, there is a relation between the limit set of $x(n)$ and the equilibria of the vector field.
For graph-based P\'olya's urns, the vector field is gradient-like \cite[Lemma 4.1]{benaim2013generalized},
thus the limit set of $x(n)$ is almost surely contained in the equilibria set of the vector field.

\medskip
Since limit sets are connected, if there are finitely many equilibria then $x(n)$ converges almost surely to
some equilibrium.
Some of them are unstable, and some are not (see \S\ref{section-saa}).
The probability that $x(n)$ converges to an unstable equilibrium is zero \cite[Lemma 5.2]{benaim2013generalized},
see also Theorem \ref{theorem-zero-probability-unstable} here. Hence at least one equilibrium
is non-unstable. Complementary to this, non-unstable equilibria generate Lyapunov functions
\cite[Lemmas 3.1 and 3.2]{chen2013generalized}. If $G$ is not balanced bipartite,
this implies that there is at most one non-unstable equilibrium.
Combined, these two arguments imply the second part of Corollary \ref{main-corollary}
\cite[Theorem 1.1]{chen2013generalized}.

\medskip
If there are infinitely many non-unstable equilibria, then $x(n)$ could wander around without converging
to any of them. To prove convergence, one needs to understand the attracting/repelling properties of the equilibria.
E.g. if $G$ is regular and balanced bipartite then the set of non-unstable equilibria is an interval and the eigenvalues
in transverse directions have negative real part \cite[Lemma 10.1]{benaim2013generalized}, thus
$x(n)$ converges almost surely to a point of the interval \cite[Theorem 1.2]{chen2013generalized}.
Here is a heuristic explanation: the orbits of the vector field converge exponentially fast to the interval,
and the random model converges exponentially fast to its limit set \cite[Lemma 4.1]{chen2013generalized}.
This prevents $x(n)$ of wandering around the interval.

\medskip
We prove that for balanced bipartite graphs the set of non-unstable equilibria is an
interval $\mathfs J$, possibly reduced to a point. When it is not a point, we prove that
all eigenvalues are real, and those in transverse directions to $\mathfs J$ are negative.
Under these conditions, we apply the methods of \cite[Theorem 1.2]{chen2013generalized}
to prove that $x(n)$ converges to a point of $\mathfs J$ almost surely. This gives
Theorem \ref{main-theorem}.

\section{Stochastic approximation algorithms}\label{section-saa}

Graph-based P\'olya's urn are an example of stochastic approximation algorithms \cite{benaim2013generalized}.
In this section we recall some previous results of \cite{benaim2013generalized,chen2013generalized} and explain
how a graph-based P\'olya's urn is related to a vector field.

\medskip
\noindent
{\sc Stochastic approximation algorithm:} A {\em stochastic approximation algorithm} is a discrete time process $\{x(n)\}_{n\geq 0}\subset\mathbb R^m$
of the form
\begin{align}\label{stochastic-approximation-algorithm}
x(n+1)-x(n)=\gamma_n\left[F(x(n))+u_n\right]
\end{align}
where $\{\gamma_n\}_{n\ge 0}$ is a sequence of nonnegative scalar gains, $F:\mathbb R^m\to\mathbb R^m$ is a vector field,
and $u_n\in\mathbb R^m$ is a random vector that depends on $x(n)$ only.

\medskip
\noindent
Let $\mathfs F_n$ be the sigma-algebra generated by the process up to step $n$. Since $u_n$ only depends on $x(n)$ we can assume,
after changing $F$, that $\EE{u_n}{\mathfs F_n}=0$.

\medskip
Graph-based P\'olya's urns are stochastic approximation algorithms with $\gamma_n=\frac{1}{\frac{N_0}{N}+(n+1)}$
and vector field $F$ defined by the equations:
\begin{eqnarray}\label{vector-field}
\left\{
\begin{array}{rcl}
\dfrac{dv_1}{dt}&=&-v_1+\dfrac{1}{N}\displaystyle\sum_{j\sim 1}\frac{v_1}{v_1+v_j} \\
&\vdots&\\
\dfrac{dv_m}{dt}&=&-v_m+\dfrac{1}{N}\displaystyle\sum_{j\sim m}\frac{v_m}{v_m+v_j}\,\cdot
\end{array}\right.
\end{eqnarray}
See \cite[\S3.2]{benaim2013generalized}.

\medskip
\noindent
{\sc Domain of $F$:} Fix $c<\frac{1}{N}$,
and let $\Delta$ be the set of vectors $(v_1,\ldots,v_m)\in\mathbb R^m_{\geq 0}$ with
$\sum_{i=1}^m v_i=1$ and $v_i+v_j\ge c$ for all $\{i,j\}\in E$.
The vector field $F:\Delta\rightarrow T\Delta$ is  Lipschitz, and it induces a semiflow \cite[Lemma 2.1]{benaim2013generalized}.

\subsection*{The vector field $F$ is gradient-like}

This was proved in \cite[Lemma 4.1]{benaim2013generalized}.

\medskip
\noindent
{\sc Equilibria set:} $v\in\Delta$ is called an {\em equilibrium} if $F(v)=0$. $v$ is called {\it unstable} if
$DF(v)$ has an eigenvalue with negative real part, and {\em non-unstable} otherwise.
The {\em equilibria set} is $\Lambda=\{v\in\Delta:v\text{ is equilibrium}\}$.

\medskip
\noindent
{\sc Lyapunov function:} Let $U\subseteq\Delta$. A continuous map $L:\Delta\to\mathbb R$ is called a
{\it Lyapunov function for $U$} if it is strictly monotone along any integral curve of $F$ outside $U$.
If $U=\Lambda$, we call $L$ a {\em strict Lyapunov function} and $F$ {\em gradient-like}.

\medskip
Let $L:\Delta\rightarrow\mathbb R$ be the function
\begin{align}\label{definition-strict-lyapunov-function}
L(v_1,\ldots,v_m)=-\sum_{i=1}^m v_i+\dfrac{1}{N}\sum_{\{i,j\}\in E}\log{(v_i+v_j)}.
\end{align}
$L$ is a strict Lyapunov function for $F$: because $\frac{dv_i}{dt}=v_i\frac{\partial L}{\partial v_i}$,
then
$$
\frac{d}{dt}(L\circ v)=\sum_{i=1}^m \frac{\partial L}{\partial v_i}\frac{dv_i}{dt}=\sum_{i=1}^m v_i\left(\frac{\partial L}{\partial v_i}\right)^2\geq 0.
$$
Equality holds iff $v_i\frac{\partial L}{\partial v_i}=0$ for all $i$ iff $v\in\Lambda$.

\medskip
We divide the singularities according to the faces of $\Delta$. Given $S\subseteq[m]$, let
$\Delta_S=\{v\in\Delta:v_i=0\text{ iff }i\notin S\}$. The restriction $F\restriction_{\Delta_S}$ is a semiflow.
Let $\Lambda_S=\{v\in\Delta_S:\frac{\partial L}{\partial v_i}(v)=0,\forall i\in S\}$.
A direct calculation shows that $\Lambda=\bigcup_{S\subseteq[m]}\Lambda_S$ \cite[Lemma 2.1]{benaim2013generalized}.
Because $L$ is a concave function, so is $L\restriction_{\Delta_S}$, hence $\Lambda_S$ is
the set of maxima of $L\restriction_{\Delta_S}$.

\subsection*{Relation between $\{x(n)\}_{n\geq 0}$ and $F$}

Let $\{\Phi_t\}_{t\geq 0}$ be the semiflow induced by $F$. Let $\tau_n=\sum_{i=0}^n\gamma_i$, and let $\{X(t)\}_{t\geq 0}$ be the interpolation of $\{x(n)\}_{n\geq 0}$: $X(\tau_n)=x(n)$ and $X\restriction_{[\tau_n,\tau_{n+1}]}$ is linear.
Let $d$ be the euclidean distance on $\Delta$.

\begin{theorem}\cite{benaim2013generalized,chen2013generalized}\label{theorem-random-versus-deterministic}
The limit set of $x(n)$ is contained in $\Lambda$ almost surely, and
\begin{equation}\label{estimate-convergence}
\sup_{T>0}\limsup_{t\to+\infty}\frac{1}{t}\log\left(\sup_{0\leq h\leq T}d\big(X(t+h),\Phi_h(X(t))\big)\right)\leq -\frac{1}{2}\cdot
\end{equation}
\end{theorem}

\medskip
The first part was proved in \cite[\S3.1 and \S3.2]{benaim2013generalized}. It is an application of the general theory of
stochastic approximation algorithms \cite{benaim1996dynamical,benaim1999dynamics}.
The second part is \cite[Lemma 4.1]{chen2013generalized}.
It follows from shadowing techniques that relate the speed of convergence of the interpolated process and
the vector field \cite[Prop. 8.3]{benaim1999dynamics}. The right hand side of the inequality
is the log-convergence rate $\frac{1}{2}\limsup\frac{\log\gamma_n}{\tau_n}$, which for graph-based P\'olya's urns equals $-\frac{1}{2}$.

\section{Unstable and non-unstable equilibria}\label{section-equilibria}

Write $F=(F_1,\ldots,F_m)$, $F_i=v_i\frac{\partial L}{\partial v_i}$.
Fix $w\in\Lambda_S$, and let $DF(w):T_w\Delta\to T_w\Delta$.
In coordinates $v_1,\ldots,v_m$, $DF(w)$ equals the jacobian matrix
$JF(w)=\left(\frac{\partial F_i}{\partial v_j}\right)_{i,j}$:
\begin{equation}\label{jacobian}
\frac{\partial F_i}{\partial v_j}=\left\{\begin{array}{ll}
v_i\dfrac{\partial ^2L}{\partial v_i\partial v_j}&\text{if }i\not=j,\\
&\\
\dfrac{\partial L}{\partial v_i}+v_i\dfrac{\partial ^2L}{\partial v_i^2} &\text{if }i= j.\\
\end{array}\right.
\end{equation}
Without loss of generality, assume that $S=\{k+1,\ldots, m\}$. Thus
\begin{align}\label{definition jacobian}
JF(w)=\left[
\begin{array}{cc}
A & 0\\
C & B \\
\end{array}
\right]
\end{align}
where $A$ is a $k\times k$ diagonal matrix with $a_{ii}=\frac{\partial L}{\partial v_i}(w)$, $i\in[k]$.

\subsection*{Non-convergence to unstable equilibria}

The spectrum of $JF(w)$ is the union of the spectra of $A$ and $B$. Introduce the inner
product $(x,y)=\sum_{i=k+1}^m x_iy_i/v_i$. $B$ is self-adjoint and negative semidefinite
(by the concavity of $L$), hence the eigenvalues of $B$ are real and nonpositive.
Therefore $JF(w)$ has a real positive eigenvalue iff $a_{ii}>0$ for some $i\in[k]$.
In summary:
\begin{equation}\label{condition-unstable-equilibrium}
w\in\Lambda_S\text{ is unstable}\iff\exists i\not\in S\text{ s.t. }\tfrac{\partial L}{\partial v_i}(w)>0.
\end{equation}

\begin{theorem}\cite[Lemma 5.2]{benaim2013generalized}\label{theorem-zero-probability-unstable}
If $w$ is an unstable equilibrium, then
$$
\mathbb P\left[\lim_{n\to\infty}x(n)=w\right]=0.
$$
\end{theorem}

\medskip
In particular, if $\Lambda$ is finite then $x(n)$ almost surely converges to a non-unstable equilibrium.
The proof is probabilistic and follows the lines of \cite[\S3 and \S4]{pemantle1992vertex}, see also \cite[\S9]{benaim1999dynamics}.

\subsection*{Non-unstable equilibria and Lyapunov functions}

Let $w\in\Lambda_S$  non-unstable. By (\ref{condition-unstable-equilibrium}),
$\frac{\partial L}{\partial v_i}(w)\leq 0$ for every $i\not\in S$. Since $\frac{\partial L}{\partial v_i}(w)=0$
for $i\in S$, we have \cite{chen2013generalized}:
\begin{equation}\label{characterization-non-unstable}
w\in \Lambda_S\text{ is non-unstable}\iff \tfrac{\partial L}{\partial v_i}(w)=0\ \ \forall i\in S\text{, and }
\tfrac{\partial L}{\partial v_i}(w)\leq 0\ \ \forall i\not\in S.
\end{equation}
In particular, every $w\in\Lambda_{[m]}$ is non-unstable.

\medskip
For every non-unstable equilibrium there is a Lyapunov function that gives extra information on the convergence of the
vector field \cite{chen2013generalized}. This fact will be used in \S4 and \S5, thus we state it
in a general form. Given $w\in\Delta_S$ and $\chi\in (0,\min_{i\in S}w_i]$,
let $\Delta^{w,\chi}=\{v\in \Delta:v_i\geq\chi,\forall i\in S\}$ (we do not require that $v_i=0$ for $i\notin S$).
$\Delta^{w,\chi}$ is a closed convex set that contains $w$ at its boundary. The next result is a summary of \cite[Lemmas 3.1 and 3.2]{chen2013generalized}.

\begin{lemma}\label{lemma-non-unstable-versus-lyapunov-functions}
Let $w\in\Lambda_S$ non-unstable. Then there is a closed interval $J=J(w,\chi)$
such that $H:\Delta^{w,\chi}\to\mathbb R$,
$H(v)=\sum_{i\in S}w_i\log v_i$, is a Lyapunov function for $J$.
\end{lemma}

In particular, every orbit of $F\restriction_{\Delta^{w,\chi}}$ converges to $J$.

\begin{proof}
Inside $\Delta^{w,\chi}$ the function
$H$ is differentiable, and
$$
\frac{d}{dt}(H\circ v)=\sum_{i\in S}w_i\frac{1}{v_i}\frac{dv_i}{dt}=\sum_{i\in S}w_i\frac{\partial L}{\partial v_i}
=\sum_{i=1}^m w_i\frac{\partial L}{\partial v_i}=-1+\frac{1}{N}\sum_{\{i,j\}\in E}\frac{w_i+w_j}{v_i+v_j}\cdot
$$

Let $f:\Delta^{w,\chi}\to\mathbb R$, $f(v)=-1+\frac{1}{N}\sum_{\{i,j\}\in E}\frac{w_i+w_j}{v_i+v_j}$.
Observe that $f(w)=0$. We will show that $f(v)\geq 0$, with equality iff $v\in J$ (to be defined below).

\medskip
\noindent
{\sc Step 1:} $f$ is convex.

\medskip
Since $x>0\mapsto \frac{1}{x}$ is convex, each $v\in\Delta^{w,\chi}\mapsto \frac{w_i+w_j}{v_i+v_j}$ is convex.
Thus $f$ is the sum of convex functions.

\medskip
\noindent
{\sc Step 2:} $w$ is a global minimum of $f$.

\medskip
Since $f$ is convex, it is enough to prove that $w$ is a local minimum of $f$.
Let $v=w+(\ve_1,\ldots,\ve_m)$ with $\ve_1,\ldots,\ve_m$ small enough.  Of course, $\ve_i\geq 0$ for $i\not\in S$.
Applying the inequality $\frac{x}{x+\ve}-1\geq -\frac{\ve}{x}$ for $x,x+\ve>0$, we have
\begin{eqnarray*}
f(v)-f(w)&=&\frac{1}{N}\sum_{\{i,j\}\in E}\left[\frac{w_i+w_j}{v_i+v_j}-1\right]\\
&\geq&\frac{1}{N}\sum_{\{i,j\}\in E}-\frac{\ve_i+\ve_j}{w_i+w_j}\\
&=&-\sum_{i=1}^m \ve_i\frac{1}{N}\sum_{j\sim i}\frac{1}{w_i+w_j}\\
&=&-\sum_{i=1}^m \ve_i\left[1+\frac{\partial L}{\partial v_i}(w)\right]\\
&=&-\sum_{i=1}^m \ve_i \frac{\partial L}{\partial v_i}(w)\\
&\geq& 0,
\end{eqnarray*}
since $\ve_i\frac{\partial L}{\partial v_i}(w)=0$ for $i\in S$, and $\ve_i\frac{\partial L}{\partial v_i}(w)\leq 0$
for $i\not\in S$. Hence $w$ is a local minimum of $f$.

\medskip
\noindent
{\sc Step 3:} The set of global minima of $f$ is a closed interval $J\ni w$.

\medskip
The set of global minima of a convex function is convex. Thus if $v\in\Delta^{w,\chi}$ with $f(v)=f(w)$
then $f(tv+(1-t)w)=tf(v)+(1-t)f(w)$ for all $t\in[0,1]$. Because $x>0\mapsto\frac{1}{x}$ is strictly convex,
we get $v_i+v_j=w_i+w_j$ for all $\{i,j\}\in E$, i.e.
\begin{equation}\label{condition-equality}
v_i-w_i=-(v_j-w_j),\ \forall \{i,j\}\in E.
\end{equation}
We divide the analysis of (\ref{condition-equality}) into three cases:
\begin{enumerate}[$\circ$]
\item $G$ is not bipartite: $G$ has an odd cycle, thus (\ref{condition-equality}) implies $v=w$. Take $J=\{w\}$.
\item $G$ is unbalanced bipartite: let $V=A\cup B$ be the bipartition, $\#A\neq\#B$.
By (\ref{condition-equality}), there is $\eta\in\mathbb R$ such that
\begin{equation}\label{condition-equality-2}
v_i=\left\{
\begin{array}{ll}
w_i+\eta & \text{ if }i\in A,\\
w_i-\eta & \text{ if }i\in B.
\end{array}
\right.
\end{equation}
Summing up on $i$, we get $\eta(\#A-\#B)=0\Rightarrow\eta=0$. Take $J=\{w\}$.
\item $G$ is balanced bipartite: let $V=A\cup B$ be the bipartition, $\#A=\#B$. As in the previous case,
(\ref{condition-equality-2}) holds. Take $J=\{v\in\Delta^{w,\chi}:v\text{ satisfies (\ref{condition-equality-2})}\}$.
$J$ is a closed interval, and $f\restriction_J$ is identically zero.
\end{enumerate}
\end{proof}

We want to avoid the dependence of $J$ on $w,\chi$.

\medskip
\noindent
{\sc The interval $\mathfs J$:} $\mathfs J$ is the maximal extension of $J$ to $\Delta$.

\medskip
$\mathfs J$ is an interval whose endpoints belong to $\partial\Delta$, one of which is $w$,
and whose interior is contained in $\Delta_{[m]}$. Furthermore:
\begin{enumerate}[i)]
\item[(i)] $\mathfs J$ is uniquely determined by any of its points.
\item[(ii)] $\frac{\partial L}{\partial v_i}\restriction_{\mathfs J}$ is constant and equal to $\frac{\partial L}{\partial v_i}(w)$
for all $i$, because of (\ref{condition-equality-2}).
\end{enumerate}


\section{Not balanced bipartite graphs}

If $G$ is not balanced bipartite, then Corollary \ref{main-corollary}
holds with $\mathfs J=$ singleton \cite[Theorem 1.1]{chen2013generalized}.
We include the proof for completeness.

\medskip
\noindent
{\sc Step 1:} $L$ is strictly concave.

\medskip
We have $L(tv+(1-t)w)\geq tL(v)+(1-t)L(w)$ for all $v,w\in\Delta,t\in[0,1]$. Equality holds iff 
(\ref{condition-equality}) holds iff $v=w$, because:
\begin{enumerate}[$\circ$]
\item If $G$ is not bipartite then it has an odd cycle, hence $v=w$.
\item If $G$ is unbalanced bipartite then (\ref{condition-equality-2}) holds, hence $v=w$.
\end{enumerate}

\medskip
\noindent
{\sc Step 2:} $\Lambda$ is finite.

\medskip
$L\restriction_{\Delta_S}$ is strictly concave, because it is the restriction
of $L$ to a convex set. Thus $\Lambda_S$ is either empty or a singleton,
and $\Lambda=\bigcup_{S\subseteq[m]}\Lambda_S$ is finite. 

\medskip
\noindent
{\sc Step 3:} There is at least one non-unstable equilibrium.

\medskip
This follows directly from Theorem \ref{theorem-zero-probability-unstable}.

\medskip
\noindent
{\sc Step 4:} There is at most one non-unstable equilibrium.

\medskip
Suppose $w\neq\widetilde w$ are non-unstable equilibria. 
Let $H:\Delta^{w,\chi}\to\mathbb R$, $\widetilde H:\Delta^{\widetilde w,\chi}\to\mathbb R$
as in Lemma \ref{lemma-non-unstable-versus-lyapunov-functions}.
Take $\chi>0$ small enough such that $\Delta^{w,\chi}\cap\Delta^{\widetilde w,\chi}\neq\emptyset$.
Every orbit of $F$ starting from $\Delta^{w,\chi}\cap\Delta^{\widetilde w,\chi}$ converges
simultaneously to $w$ and $\widetilde w$, a contradiction.

\medskip
By steps 3 and 4 there is a unique non-unstable equilibrium $w=w(G)$, and
$x(n)$ converges to $w$ almost surely.

\section{Balanced bipartite graphs}\label{section-bipartite}

Let $V=A\cup B$ be the bipartition, $\#A=\#B$. We consider two cases.

\subsection*{First case: $\Lambda_{[m]}=\emptyset$}

Steps 1--3 below are in \cite[Corollary 5.2]{chen2013generalized}.

\medskip
\noindent
{\sc Step 1:} $L\restriction_{\Delta_S}$ is strictly concave for every $S\neq[m]$.

\medskip
If $L(tv+(1-t)w)=tL(v)+(1-t)L(w)$ with $v,w\in\Delta_S,t\in[0,1]$, then (\ref{condition-equality-2}) holds.
For $i\in[m]\backslash S$ we have $v_i=w_i=0$, hence $\eta=0$.

\medskip
\noindent
{\sc Step 2:} $\Lambda$ is finite.

\medskip
By step 1, if $S\neq[m]$ then $\Lambda_S$ is either empty or a singleton. Since
$\Lambda_{[m]}=\emptyset$, $\Lambda=\bigcup_{S\subseteq [m]}\Lambda_S$ is finite.

\medskip
\noindent
{\sc Step 3:} There is at least one non-unstable equilibrium.

\medskip
Again, this is consequence of Theorem \ref{theorem-zero-probability-unstable}.

\medskip
\noindent
{\sc Step 4:} There is at most one non-unstable equilibrium.

\medskip
Let $w\neq\widetilde w$ be non-unstable equilibria, let $\Delta^{w,\chi},\Delta^{\widetilde w,\chi}$
as in Lemma \ref{lemma-non-unstable-versus-lyapunov-functions}, and
$\mathfs J,\widetilde{\mathfs J}$ be the maximal intervals defined at the end of \S3.
Choose $\chi>0$ small enough so that $\Delta^{w,\chi}\cap\Delta^{\widetilde w,\chi}\neq\emptyset$.
Every orbit of $F$ starting from $\Delta^{w,\chi}\cap\Delta^{\widetilde w,\chi}$ converges to both $\mathfs J$
and $\widetilde{\mathfs J}$. Since $F$ is gradient-like they also converge to $\Lambda$,
thus $\mathfs J\cap\widetilde{\mathfs J}\cap\Lambda\not=\emptyset$.
This will give the contradiction we are looking for.

\medskip
Since $\mathfs J\cap\widetilde{\mathfs J}\not=\emptyset$ and $\mathfs J,\widetilde{\mathfs J}$
are determined by any of its points, $\mathfs J=\widetilde{\mathfs J}$.
$w$ is an endpoint of $\mathfs J$, and $\widetilde w$ is and endpoint
of $\widetilde{\mathfs J}$, thus $w,\widetilde w$ are the two endpoints of
$\mathfs J=\widetilde{\mathfs J}$. In particular, if $w_i=0$ then $\widetilde w_i>0$.
This gives that $([m]\backslash S)\cap([m]\backslash\widetilde S)=\emptyset$,
hence $S\cup\widetilde S=V$. By (\ref{characterization-non-unstable}) we get
$\mathfs J\subseteq\Lambda$: if $v\in\mathfs J$ then $\frac{\partial L}{\partial v_i}(v)=\frac{\partial L}{\partial v_i}(w)=0$
for $i\in S$, and $\frac{\partial L}{\partial v_i}(v)=\frac{\partial L}{\partial v_i}(\widetilde w)=0$ for
$i\in \widetilde S$.
In particular $\emptyset\neq {\rm int}(\mathfs J)\subset\Lambda_{[m]}$, a contradiction.

\medskip
By steps 3 and 4, there is a unique non-unstable equilibrium $w=w(G)$, and
$x(n)$ converges to $w$ almost surely.

\subsection*{Second case: $\Lambda_{[m]}\not=\emptyset$}

We will prove that there is a non-degenerate interval $\mathfs J=\mathfs J(G)$ such that
$x(n)$ converges to a point of $\mathfs J$ almost surely.

\medskip
\noindent
{\sc Step 1:} The set on non-unstable equilibria is a closed interval $\mathfs J$. 

\medskip
Remember that any $w\in\Lambda_{[m]}$ is non-unstable, since $\frac{\partial L}{\partial v_i}(w)=0$
for all $i$. Apply Lemma \ref{lemma-non-unstable-versus-lyapunov-functions} to $w$,
and let $\mathfs J$ be the maximal interval defined as in the end of \S3.
$\frac{\partial L}{\partial v_i}\restriction_{\mathfs J}$ is identically zero for all $i$, hence
$\mathfs J$ is an interval of non-unstable equilibria.

\medskip
We now show that $\mathfs J$ is {\em the} set of all non-unstable equilibria. The proof is
similar to the proof of step 4 of the first case.
Let $\widetilde w$ be a non-unstable equilibrium, and let $\widetilde{\mathfs J}$ be the maximal
interval defined as in the end of \S3. If $\chi>0$ is sufficiently small then $J\cap\widetilde J\neq\emptyset$,
thus $\mathfs J=\widetilde{\mathfs J}$. Hence $\widetilde w\in{\mathfs J}$.

\begin{remark}\label{remark-balanced-bipartite}
Step 1 above and the first case characterize, for balanced bipartite graphs, when
$\mathfs J$ is a singleton or not.
\begin{enumerate}[$\circ$]
\item $\mathfs J$ is a singleton iff there is a non-unstable equilibrium $w$ with
$\frac{\partial L}{\partial v_i}<0$ for some $i$: otherwise $w$ would define an interval of equilibria
whose interior is a subset of $\Lambda_{[m]}$.
\item $\mathfs J$ is a non-degenerate interval iff
$\frac{\partial L}{\partial v_i}=0$, $i\in[m]$, for all non-unstable equilibria:
since $\frac{\partial L}{\partial v_i}\restriction_{\Lambda_{[m]}}\equiv 0$
and $\frac{\partial L}{\partial v_i}\restriction_{\mathfs J}$ is constant, we have
$\frac{\partial L}{\partial v_i}\restriction_{\mathfs J}\equiv 0$.
\end{enumerate}
We will make use of this in the discussion of some examples, see \S\ref{section-conclusion}.
\end{remark}

\medskip
\noindent
{\sc Step 2:} If $w\in{\rm int}(\mathfs J)$ then all eigenvalues of $DF(w)$ are real, and any eigenvalue
in a transverse direction to $\mathfs J$ is negative.

\medskip
This was proved for {\em regular} balanced bipartite graphs \cite[Lemma 10.1]{benaim2013generalized}.
The question remained open for a general balanced bipartite graph.

\medskip
Let $w\in{\rm int}(\mathfs J)$, thus $\frac{\partial L}{\partial v_i}(w)=0$ for all $i$.
By (\ref{jacobian}), $DF(w)$ is the restriction 
of the matrix $B=\left(v_i\frac{\partial^2L}{\partial v_i\partial v_j}\right)$ to $T_w\Delta$.
Let $A=\left(\frac{\partial^2L}{\partial v_i\partial v_j}\right)$
be the Hessian of $L$ in the coordinates $v_1,\ldots,v_m$.
The rows of $B$ are positive multiples of the rows of $A$.

\medskip
The matrix $A$ is symmetric, thus its eigenvalues are real.
Since ${\rm int}(\mathfs J)$ is the set of global maxima of $L\restriction_{\Delta_{[m]}}$,
$A$ is negative semidefinite and zero is a simple eigenvalue, i.e. every eigenvalue in a
transverse direction to $\mathfs J$ is negative. We claim that the same is true for $B$.
Remind the inner product $(x,y)=\sum_{i=1}^m \frac{x_iy_i}{v_i}$ introduced in \S\ref{section-equilibria},
and let $\langle x,y\rangle=\sum_{i=1}^m x_iy_i$ be the canonical inner product. 

\medskip
Since $(Bx,y)=\langle Ax,y\rangle$, $B$ is self-adjoint:
$(Bx,y)=\langle Ax,y\rangle=\langle x,Ay\rangle=(x,By)$.
Thus the eigenvalues of $B$ are real. Let us prove that one of them is zero and the others
are negative.
\begin{enumerate}[$\circ$]
\item 0 is a simple eigenvalue: $B=DA$, where $D$ is the diagonal matrix with diagonal entries $v_1,\ldots,v_m$.
Since $v\in\Delta_{[m]}$, $D$ is invertible, thus the kernels of $A$ and $B$ coincide. In particular, the kernel of $B$ is one-dimensional.
\item 0 is the largest eigenvalue: let $M=\max_{i\in[m]}v_i>0$, thus $(x,x)\geq M^{-1}\langle x,x\rangle$.
Let $\lambda_1(\cdot)$ denote the largest eigenvalue of a matrix. By the variational characterization of
eigenvalues of hermitian matrices (see \cite[Theorem 4.2.2]{horn1985matrix}),
$$
\lambda_1(B)=\max_{x\neq 0}\frac{(Bx,x)}{(x,x)}\leq M\max_{x\neq 0}\frac{\langle Ax,x\rangle}{\langle x,x\rangle}=
M\lambda_1(A)=0.
$$
\end{enumerate}
 This concludes the proof of step 2.

\medskip
\noindent
{\sc Step 3:} $x(n)$ converges to a point of $\mathfs J$ almost surely.

\medskip
It is enough to prove that the interpolated orbits $X(t)$ converge to a point of $\mathfs J$ almost surely.
This is true for regular balanced bipartite graphs \cite[Theorem 1.2]{chen2013generalized}.
Here is a heuristic of the proof: since the interpolated process converges exponentially fast
(Theorem \ref{theorem-random-versus-deterministic}) and the orbits of $F$ also converge exponentially
fast (step 2), the interpolated process cannot wander around ${\mathfs J}$. Provided
these are true, the proof in \cite{chen2013generalized}
applies ipsis literis. We include it for completeness.

\medskip
For a fixed closed interval $I\subset {\rm int}(\mathfs J)$, and
a small neighborhood $U$ of $I$ in $\Delta$, there is a foliation $\{\mathfs F_x\}_{x\in U}$ such that:
\begin{enumerate}[$\circ$]
\item $\mathfs F_x$ is a submanifold with $\mathfs F_x \pitchfork \mathfs J$ at a single point $\pi(x)$.
\item $\pi(x)$ is a hyperbolic attractor for $F\restriction_{\mathfs F_x}$.
The speed of convergence depends on the negative eigenvalues of $DF(\pi(x))$.
\end{enumerate}
This is an application of the theory of invariant manifolds for normally hyperbolic sets, see
\cite[Theorem 4.1]{hirsch1977invariant}.

\medskip
The map $\pi:U\to \mathfs J$ is not necessarily a projection (it is not even linear),
but since $\mathfs F_x$ depends smoothly on $x$, if $U$ is small enough then $\pi$ is 2--Lipschitz: 
\begin{equation}\label{pi-lipschitz}
d(\pi(x),\pi(y))\leq 2d(x,y),\forall\,x,y\in U.
\end{equation}
Fix a small parameter $\varepsilon>0$ and reduce $U$, if necessary, so that
\begin{equation}\label{definition-neighborhood}
U=\{x\in\Delta:\pi(x)\in I\text{ and }d(x,\pi(x))<\varepsilon\}.
\end{equation}
Let $c=\max\{\lambda:\lambda\not=0\text{ is eigenvalue of }DF(x), x\in I\}$.
By step 2, $c<0$. Thus there is $K>0$ such that
\begin{align}\label{estimate-hyperbolicity}
d(\Phi_t(x),\pi(x))\le K e^{ct}d(x,\pi(x)), \forall\,x\in U,\forall\,t\ge 0.
\end{align}
(Remind: $\{\Phi_t\}_{t\geq 0}$ is the semiflow induced by $F$.)

\medskip
Fix an interpolated orbit $X(t)$ that does not converge to the endpoints of $\mathfs J$.
It has an accumulation point in ${\rm int}(\mathfs J)$.
Let $I\subset {\rm int}(\mathfs J)$ be an interval containing such point, and let
$U$ as in (\ref{definition-neighborhood}).

\begin{lemma}\cite[Lemma 4.4]{chen2013generalized}\label{lemma-iteration}
Assume that $X(t)\in U$. If $t,T$ are large enough, then
\begin{enumerate}[i)]
\item[{\rm (i)}] $d(\pi(X(t+T)),\pi(X(t)))<2e^{-\frac{t}{4}}$.
\item[{\rm (ii)}] $X(t+T)\in U$.
\end{enumerate}
\end{lemma}

\begin{proof} To simplify the notation, denote $X(t)$ by $X$ and $X(t+T)$ by
$X(T)$.

\medskip
\noindent
(i) Since $\pi(\Phi_T(X))=\pi(X)$ and $\pi$ is 2--Lipschitz,
\begin{align*}
d(\pi(X(T)),\pi(X))=d(\pi(X(T)),\pi(\Phi_T(X)))\le
2d(X(T),\Phi_T(X)).
\end{align*}
By (\ref{estimate-convergence}), $d(X(T),\Phi_T(X))<e^{-\frac{t}{4}}$ for large $t$, therefore 
$d(\pi(X(T)),\pi(X))<2e^{-\frac{t}{4}}$ for large $t$.

\medskip
\noindent
In particular, $\pi(X(T))\in I$ for large $t$.

\medskip
\noindent
(ii) Since $\pi(X(T))\in I$, it remains to estimate $d(X(T),\pi(X(T)))$. By the triangular inequality, (\ref{pi-lipschitz}) and (\ref{estimate-hyperbolicity}), we have
\begin{eqnarray*}
d(X(T),\pi(X(T)))
&\le&d(X(T),\Phi_T(X))+d(\Phi_T(X),\pi(\Phi_T(X))) \nonumber \\
&&+d(\pi(\Phi_T(X)),\pi(X(T))) \nonumber \\
&\le&3d(X(T),\Phi_T(X))+d(\Phi_T(X),\pi(X)) \nonumber \\
&\le&3e^{-\frac{t}{4}}+Ke^{cT}d(X,\pi(X)) \nonumber \\
&\le&3e^{-\frac{t}{4}}+Ke^{cT}\varepsilon \nonumber \\
&<&\varepsilon
\end{eqnarray*}
provided $3e^{-\frac{t}{4}}<\frac{\varepsilon}{2}$ and $Ke^{cT}<\frac{1}{2}$.
\end{proof}

\medskip
The second part of the lemma allows us to apply it inductively to the points
$X_k:=X(t+kT),k\geq 0$. For that, choose $t,T$ large enough so that
$2\sum_k e^{-\frac{t+kT}{4}}<d(\pi(X_0),\mathfs J\backslash I)$. By Lemma \ref{lemma-iteration},
if $X_k\in U$ then $X_{k+1}\in U$ and $d(\pi(X_{k+1}),\pi(X_k))$ $<2e^{-\frac{t+kT}{4}}$.
Thus $\pi(X_k)$ converges, say $\lim \pi(X_k)=\widetilde x$.

\medskip
The proof of Lemma \ref{lemma-iteration}(ii) also gives that
\begin{equation*}
 d(X_k,\pi(X_k))\leq 3e^{-\frac{t+(k-1)T}{4}}+Ke^{cT}d(X_{k-1},\pi(X_{k-1})), \ k\geq 1.
\end{equation*}
Let $\lambda=Ke^{cT}$, thus:
\begin{eqnarray*}
 d(X_k,\pi(X_k))&\leq& 3e^{-\frac{t}{4}}\left( e^{-\frac{(k-1)T}{4}}+\lambda e^{-\frac{(k-2)T}{4}} +\cdots+\lambda^{k-1}\right)+\lambda^k d(X_0,\pi(X_0))  \nonumber  \\
&\le& 3e^{-\frac{t}{4}}k \left(\max{ \left\{e^{-\frac{T}{4}},\lambda\right\}} \right)^{k-1}+\lambda^k d(X_0,\pi(X_0)).  \nonumber  
\end{eqnarray*}
When $T$ is large, $\max{ \{e^{-\frac{T}{4}},\lambda\}} <1$, hence $\lim d(X_k,\pi(X_k))=0$. 
Since $d(X_k,\widetilde x)\leq d(X_k,\pi(X_k)) +d(\pi(X_k),\widetilde x)$, it follows that
$\lim X_k=\widetilde x\in I$.

\medskip
Now let $s\in[t+kT,t+(k+1)T)$. By the triangular inequality and (\ref{estimate-convergence})
\begin{eqnarray*}
d(X(s),\widetilde x)&=&d(X(s),\Phi_{s-(t+kT)}(\widetilde x))\\
&\leq&d(X(s),\Phi_{s-(t+kT)}(X_k))+d(\Phi_{s-(t+kT)}(X_k),\Phi_{s-(t+kT)}(\widetilde x))\\
&\leq&e^{-\frac{t+kT}{4}}+c(T)d(X_k,\widetilde x),
\end{eqnarray*}
where $c(T)>0$ is the supremum of the Lipschitz constants of $\Phi_\delta,\delta\in[0,T]$.
Therefore $X(t)$ converges to $\widetilde x$.

\section{Concluding remarks}\label{section-conclusion}

\subsection*{Some examples}

Consider the graphs in Figure \ref{figure-examples}.
We show that all cases considered in the proof of Corollary \ref{main-corollary}
occur. Remind: every $v\in\Lambda_{[m]}$ is non-unstable.

\begin{figure}[hbt!]
\centering
\def\svgwidth{12cm}
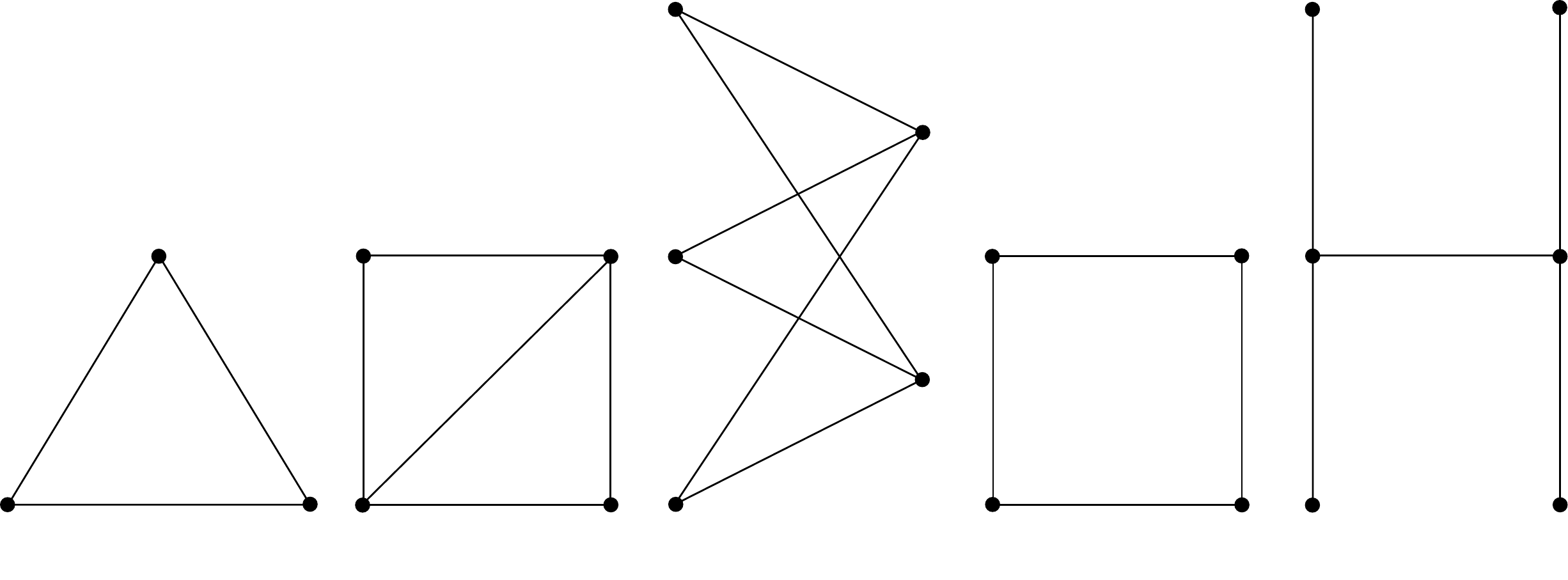
\caption{(a)--(b) are not bipartite; (c) is unbalanced bipartite; (d)--(e) are balanced bipartite.}\label{figure-examples}
\end{figure}

\medskip
\noindent (a) The triangle is not bipartite and $(\frac{1}{3},\frac{1}{3},\frac{1}{3})\in\Lambda_{[m]}$,
thus $\mathfs J=\{(\frac{1}{3},\frac{1}{3},\frac{1}{3})\}$.
Similarly, complete graphs and cycles of odd length satisfy $\mathfs J=$ uniform distribution.

\medskip
\noindent (b) The graph is not bipartite and $(0,\frac{1}{2},0,\frac{1}{2})$ is a non-unstable
equilibrium (since $\frac{\partial L}{\partial v_1}=-\frac{1}{5}$),
thus $\mathfs J=\{(0,\frac{1}{2},0,\frac{1}{2})\}$.


\medskip
\noindent (c) The graph is unbalanced bipartite and $(0,0,0,\frac{1}{2},\frac{1}{2})$ is a non-unstable
equilibrium (since $\frac{\partial L}{\partial v_1}=-\frac{1}{3}$),
thus $\mathfs J=\{(0,0,0,\frac{1}{2},\frac{1}{2})\}$. More generally, if $K_{i,j}$ is the complete bipartite
graph and if $i>j$, then $\mathfs J=\{(0,\ldots,0,\frac{1}{j},\ldots,\frac{1}{j})\}$.


\medskip
\noindent (d) The square is balanced bipartite and $(\frac{1}{4},\frac{1}{4},\frac{1}{4},\frac{1}{4})\in\Lambda_{[m]}$,
thus $\mathfs J=\{(p,q,p,q)\in\Delta\}$. A similar argument is true for any cycle of even length.

%

\medskip
\noindent (e) The graph is balanced bipartite and $(0,0,0,0,\frac{1}{2},\frac{1}{2})$
is a non-unstable equilibrium, since $\frac{\partial L}{\partial v_1}=-\frac{3}{5}$.
By Remark \ref{remark-balanced-bipartite}, $\mathfs J=\{(0,0,0,0,\frac{1}{2},\frac{1}{2})\}$.

%
%
%
%
%
%

\subsection*{Future directions}

The model introduced in \cite{benaim2013generalized} is more general
than that defined in (\ref{transition-probability}): fix $\alpha>0$ and update the bins
according to the rule
\begin{align*}
\P{i\text{ is chosen among }\{i,j\}\text{ at step }n}=\dfrac{B_i(n-1)^\alpha}{B_i(n-1)^\alpha+B_j(n-1)^\alpha}\,\cdot
\end{align*}

\medskip
If $\alpha<1$ then there is $w=w(G)$ such that $x(n)$ converges to $w$ almost surely
\cite[Theorem 1.4]{benaim2013generalized}. For $\alpha=1$ the present note and
\cite{benaim2013generalized,chen2013generalized} establish convergence.

%

\medskip
\noindent
{\sc Question 1:} If $\alpha=1$ and $G$ is balanced bipartite, what is the distribution 
of the limit of $x(n)$?

\medskip
In classical P\'olya's urn the limit has a beta distribution, see \cite[Thm 2.1]{pemantle2007survey}.

\medskip
\noindent
{\sc Question 2:} For $\alpha>1$, does $x(n)$ converge almost surely?

\medskip
\noindent
{\sc Question 3:} For hypergraph-based P\'olya's urns \cite[\S9.2]{benaim2013generalized},
does $x(n)$ converge almost surely?

\bibliographystyle{amsalpha}
\bibliography{polya-urn-linear}

\end{document}